\documentclass[11pt,reqno,centertags]{amsart}
\usepackage{hyperref}
\hypersetup{
    colorlinks,
    citecolor=blue,
    filecolor=blue,
    linkcolor=blue,
    urlcolor=blue
}
\usepackage{amsmath, amsfonts, amssymb}
\usepackage{amsthm}
\usepackage{geometry}
\newtheorem{theorem}{Theorem}
\newtheorem{lemma}{Lemma}
\newtheorem{assumption}{Assumption}
\newtheorem{condition}{Condition}
\newtheorem{corollary}{Corollary}
\newtheorem{proposition}{Proposition}
\theoremstyle{remark}
\newtheorem{remark}{Remark}

\DeclareMathOperator{\supp}{supp}
\date{}

\title[{P}oincar\'{e}-{B}eckner inequality]{A new multicomponent {P}oincar\'{e}-{B}eckner inequality}
  \author[S.~Kondratyev]{Stanislav Kondratyev}
\address[S.~Kondratyev]{CMUC, Department of
Mathematics, University of Coimbra, 3001-501 Coimbra, Portugal}{}
\email{kondratyev@mat.uc.pt}

\author[L.~Monsaingeon]{L\'{e}onard Monsaingeon}
\address[L.~Monsaingeon]{CAMGSD, Instituto Superior T\'ecnico, University of Lisbon, 1049-001 Lisboa
   Portugal}{}
\email{leonard.monsaingeon@ist.utl.pt}

\author[D.~Vorotnikov]{Dmitry Vorotnikov}
\address[D.~Vorotnikov]{CMUC, Department of
Mathematics, University of Coimbra, 3001-501 Coimbra, Portugal}{}
\email{mitvorot@mat.uc.pt}

\begin{document}
\begin{abstract}
We prove a new vectorial functional inequality of {P}oincar\'{e}-{B}eckner type. The inequality may be interpreted as an entropy-entropy production one for a gradient flow in the metric space of Radon measures. The proof uses subtle
analysis of combinations of related  super- and sub-level sets employing the coarea formula and the relative isoperimetric inequality. 
\end{abstract}
\maketitle


\vspace{10pt}

\textbf{MSC [2010] 26D10, 49Q20, 52B11, 58B20}

\section{Introduction} Let $\Omega \subset \mathbb R^d$ be a bounded, connected, open domain. Fix a vector function $\mathbf{m} \in C^1(\overline \Omega; \mathbb R^N)$ and a matrix function $A
\in C^1(\overline \Omega; M_N(\mathbb R))$.  In this paper, we contemplate the inequality \begin{equation}
\label{eq:kmpv1-int}
\int_\Omega \sum_{i=1}^N |f_{i}|^p \, \mathrm dx
\le
C
\int_\Omega \sum_{i=1}^N u_{i} (|f_{i}|^p + | \nabla f_{i} |^p) \, \mathrm dx
,
\end{equation} where $\mathbf{u} \in W_p^1(\Omega; \mathbb R^N)$ is a vector function with non-negative components $u_i\geq 0$, and $\mathbf{f}=\mathbf{f}(\mathbf{u}) =\mathbf{m}-A\mathbf{u}$. 

A quick glimpse suggests that \eqref{eq:kmpv1-int} is trivial when all  components of $\mathbf{u}$ are bounded away from zero. On the other hand, given an index set $I \subset \{1, \dots, N\}$ and a solution $\mathbf{u}$ to the linear 
system
\begin{equation}
\label{eq:vi-int}
\left\{
\begin{array}{l}
u_i = 0 \quad (i \in I)
\\
f_j = 0 \quad (j \notin I),
\end{array}
\right.
\end{equation}
inequality \eqref{eq:kmpv1-int} is clearly violated unless $I=\varnothing$. Under suitable structural assumptions on $A$ and $\mathbf{m}$ (roughly speaking, we need that \eqref{eq:vi-int} has a unique non-negative solution $\mathbf{u}_I$ for any $I$), we will show that it is enough for a function $\mathbf{u}$ to stay away from the solutions to \eqref{eq:vi-int} with $I\neq\varnothing$ in order to comply with inequality \eqref{eq:kmpv1-int}. The only solution of \eqref{eq:vi-int} compatible with the inequality is thus $$
\mathbf u^{\infty}:=\mathbf u_\varnothing=A^{-1}\mathbf m.
$$

In the case $N=1$ the only solution of \eqref{eq:vi-int} with $I\neq\varnothing$ is $\mathbf{u}\equiv 0$, thus $C$ in \eqref{eq:kmpv1-int} is expected to blow up only as $\mathbf{u}=u_1$ approaches zero in some sense. Indeed, we have recently proved in \cite{KMV15} that $C$ can be chosen in the form $1 /{\Phi(\int_\Omega u_1)}$, where $\Phi$ is a strictly increasing continuous function with $\Phi(0)=0$ (provided $N=1$, $\mathbf{m}=m_1(x)>0$ and $A\equiv 1$). The proof in \cite{KMV15} uses a generalized Beckner inequality \cite[Lemma 4]{chainais13entropy}, that is why we refer to \eqref{eq:kmpv1-int} as {P}oincar\'{e}-{B}eckner inequality. However, that proof completely fails in the multicomponent case due to implicit cross-diffusion effects. 

Our interest to \eqref{eq:kmpv1-int} comes from the fact that in the case of symmetric positive-definite matrix $A(x)$ and $p=2$ inequality \eqref{eq:kmpv1-int} is equivalent to an entropy-entropy production inequality corresponding to the gradient flow of the geodesically non-convex entropy functional \begin{equation} \label{e:15}
\mathcal E(\mathbf u)=\frac 12 \int_\Omega A(\mathbf u-\mathbf u^{\infty})\cdot(\mathbf u-\mathbf u^{\infty})
\end{equation} on the space of $N$-dimensional non-negative Radon measures on $\Omega$ equipped with the unbalanced optimal transport distance and induced Riemannian structure as recently introduced in \cite{KMV15} (see also \cite{LMS_big_2015,peyre_1_2015, LMS_small,peyre_2,MG16}).  This gradient flow coincides with a fitness-driven PDE system of population dynamics involving degenerate cross-diffusion. Inequality \eqref{eq:kmpv1-int} implies exponential convergence of the trajectories of this gradient flow to the coexistence steady-state $\mathbf u^{\infty}$ which corresponds to the so-called ideal free distribution \cite{FC69,fr72} of the populations.  We refer to our companion paper \cite{KMV16-1} for the details of this interpretation of \eqref{eq:kmpv1-int} and its implications.  

The proof of \eqref{eq:kmpv1-int} which we carry out in this paper is non-standard, being based on a subtle
analysis of suitable unions of super-level sets of the components of $\mathbf{f}$ employing the coarea formula and the relative isoperimetric inequality.  Assuming that there exists a sequence violating the inequality, either we conclude that it converges one of the degenerate states $\mathbf{u}_I$, or we can detect a drop of $f_i$ that can be exploited to estimate the total variation of $f_i$ by means of the coarea formula. To apply this consideration to the term
\begin{equation*}
\int_\Omega \sum_{i=1}^N u_{i} | \nabla f_{i} |^p \, \mathrm dx,
\end{equation*}
we must consider the variation of $f_i$ over the region where $u_i$ is not small. However, due to the hidden cross-diffusion nature of the problem,  this produces ``holes'' in the level sets of $f_i$, and we cannot use the relative isoperimetric inequality to estimate the perimeter of the super-level sets. We patch the holes by merging certain super- and sub-level sets of different $f_i$. Since we argue by contradiction, we are not able to quantify the constant $C$ in \eqref{eq:kmpv1-int}.

The paper is organized as follows. In Section \ref{sec:statement}, we give our structural conditions on $A$ and $\mathbf m$, and present the main results. In Section \ref{sec:aux}, we state some algebraic and analytical properties of $\mathbf{f}(\mathbf{u})$ and related nonlinear functions whose proofs may be found in the Appendix. In Section \ref{sec:preliminaries}, we derive the main estimates for the sequences allegedly violating \eqref{eq:kmpv1-int}. In Section \ref{sec:limit-behav-sequ}, we identify three possible scenarios which are determined by behavior of suitable combinations of super- and sub-level sets of $f_i$. The first alternative leads to the convergence to $\mathbf{u}_I$ (Section \ref{sec:convergence-case-i}). The second and the third are the most involved ones, and employ the geometric ideas described above, see Sections \ref{sec:imposs-case-ii} and \ref{sec:imposs-case-iii}. 

\section{The main results}
\label{sec:statement}

Let $\Omega \subset \mathbb R^d$ be a bounded, connected, open domain.  We 
assume that it admits the 
relative isoperimetric inequality, cf.~\cite[Remark 12.39]{Mag12}:
\begin{equation}
P(A; \Omega) \ge c_\Omega |A|^{\frac{d-1}{d}},
\quad A \subset \Omega,\ |A| \le \frac 12 | \Omega |
.
\end{equation}
Here $P(A; \Omega)$ denotes the relative perimeter of a Lebesgue measurable 
$A$ of locally finite perimeter with respect to $\Omega$.

Suppose we are given a vector function $\mathbf{m} =(m_{1}, \dots,
m_{N})\in C^1(\overline \Omega; \mathbb R^N)$ and a matrix function $A
= (a_{ij}) \in C^1(\overline \Omega; M_N(\mathbb R))$.  We assume that there 
exists $\kappa > 0$ independent of $x \in \overline \Omega$ such that

\begin{assumption}
\label{a:bounded}
We have pointwise
\begin{gather}
|a_{ij}| \le \frac 1\kappa \quad (i,j = 1, \dots, N; \; x \in \overline 
\Omega),
\label{eq:am1}
\\
m_{i} \le \frac 1\kappa \quad (i = 1, \dots, N; \; x \in \overline \Omega).
\label{eq:am2}
\end{gather}
\end{assumption}

\begin{assumption}
\label{a:det-a}
For any $ I = \{i_1, \dots, i_r\} \subset \{1, \dots, N\}$, $i_1 < \dots < 
i_r$, we have
\begin{equation}
\label{eq:am3}
\begin{vmatrix}
a_{i_1i_1} & \cdots  & a_{i_1i_r} \\
\vdots   & \ddots & \vdots \\
a_{i_ri_1} & \cdots & a_{i_ri_r}
\end{vmatrix}
\ge \kappa
.
\end{equation}
\end{assumption}

\begin{assumption}
\label{a:det-am}
For any $ I = \{i_1, \dots, i_r\} \subset \{1, \dots, N\}$, $i_1 < \dots < 
i_r$, and $j \notin I$ we have
\begin{equation}
\label{eq:am4}
\begin{vmatrix}
a_{i_1i_1} & \cdots  & a_{i_1i_r} & m_{i_1} \\
\vdots   & \ddots & \vdots & \vdots \\
a_{i_ri_1} & \cdots & a_{i_ri_r} & m_{i_r} \\
a_{ji_1} & \cdots & a_{ji_r} & m_{j}
\end{vmatrix}
\ge \kappa
.
\end{equation}
\end{assumption}

\begin{remark}
Letting $I = \varnothing$ in \eqref{eq:am4}, we see that all the functions 
$m_j$ are necessarily positive.
\end{remark}

\begin{remark}
Assumptions~\ref{a:det-a} and~\ref{a:det-am} allow for a geometrical 
interpretation, see Section~\ref{sec:aux}.
\end{remark}

\begin{remark}
For a symmetric matrix $A$, Assumption \ref{a:det-a} is equivalent to uniform positive definiteness.  However, we do not assume $A$ to be necessarily symmetric.
\end{remark}

Given a vector function $\mathbf{u} = (u_{1}, \dots, u_{N}) \colon
\overline \Omega \to \mathbb R^N$, set
\begin{equation*}
  f_i = m_i - \sum_{i = 1}^N a_{ij} u_{j} \colon \overline \Omega \to
  \mathbb R
.
\end{equation*}

\begin{theorem}
\label{th:kmpv}
Suppose that $A$ and $\mathbf m$ satisfy 
Assumptions~\ref{a:bounded}--\ref{a:det-am} and let $p \ge 1$ and $\mathbf U 
\subset W^1_p(\Omega; \mathbb R^N)$ be a set of functions such that

(i) $\mathbf u \ge 0$ for any $\mathbf u \in \mathbf U$;

(ii) no bounded with respect to the $L^p$ norm sequence $\{\mathbf u_n = 
(u_{1n}, \dots, u_{Nn})\} \subset \mathbf U$ admits a nonempty index set $I 
\subset \{1, \dots, N\}$ so that \begin{gather}
u_{in} \xrightarrow[n \to \infty]{} 0 \quad (i \in I) \quad \text{in measure}
\label{eq:kmpv3}
  \\
f_{kn} \xrightarrow[n \to \infty]{} 0 \quad (k \notin I) \quad \text{in 
measure}
\label{eq:kmpv4}
\end{gather}
Then there exists $C > 0$ such that
\begin{equation}
\label{eq:kmpv1}
\int_\Omega \sum_{i=1}^N |f_{i}|^p \, \mathrm dx
\le
C
\int_\Omega \sum_{i=1}^N u_{i} (|f_{i}|^p + | \nabla f_{i} |^p) \, \mathrm dx
\quad
(\mathbf u = (u_1, \dots, u_N) \in \mathbf U)
.
\end{equation}
\end{theorem}
\begin{remark}
The integrand in the right-hand side of \eqref{eq:kmpv1} is nonnegative, and 
the integral may be infinite.
\end{remark}
\begin{remark}
For $p>1$, by Vitali's theorem, the convergence in measure in \eqref{eq:kmpv3}, 
\eqref{eq:kmpv4} can be replaced by the convergence in $L^q$, $1 \le q < p$.
\end{remark}

Condition (ii) of Theorem~\ref{th:kmpv} means that the set $\mathbf U$ must be 
separated from a finite number of specific points in the topology of 
convergence in measure.  Specifically, it follows from 
Assumption~\ref{a:det-a} that given $I \subset \{1, \dots, N\}$, the linear 
system
\begin{equation}
\label{eq:vi}
\left\{
\begin{array}{l}
u_i = 0 \quad (i \in I)
\\
f_j = 0 \quad (j \notin I),
\end{array}
\right.
\end{equation}
has a unique solution $\mathbf{u}_I \in C^1(\overline \Omega; \mathbb R^N)$.  
It is easy to see that \eqref{eq:kmpv3} and \eqref{eq:kmpv4} are equivalent to
\begin{equation}
\label{eq:varconvergence}
\mathbf u_n \to \mathbf{u}_I \quad (n \to \infty)
\quad \text{in measure}
.
\end{equation}

Theorem~\ref{th:kmpv} admits the following stronger formulation.

Solving~\eqref{eq:vi} by Cramer's rule and recalling the assumptions, we see 
that all the functions $\mathbf u_I = (u_{I1}, \dots, u_{In})$ are bounded by 
a constant depending only on $\kappa$.  Let $M = M(\kappa)$ be an arbitrary 
number such that
\begin{equation}
  \label{eq:7.1}
  M > \sup \{ u_{Ii} \colon  I \subset \{1, \dots, N\};
  i \in \{1, \dots, N\} \}
  .
\end{equation}

\begin{theorem}
\label{th:kmpv-trunc}
Let $p \ge 1$, $\mathcal A \subset C^1(\overline \Omega; M_N(\mathbb R) \times \mathbb R^N)$, and $\mathbf U \subset W^1_p(\Omega; \mathbf R^N)$ 
be such that

(i) any $(A, \mathbf m) \in \mathcal A$ satisfies 
Assumptions~\ref{a:bounded}--\ref{a:det-am} with a constant $\kappa = 
\kappa(\mathcal A)$;

(ii) $\mathbf u \ge 0$ for any $\mathbf u \in \mathbf U$;

(iii) one cannot choose sequences $\{\mathbf u_n\} \subset \mathbf U $ and 
$\{(A_n, \mathbf m_n)\} \subset \mathcal A$ such that $\{\mathbf u_n\}$ is 
bounded in $L^p$, and \eqref{eq:kmpv3} and \eqref{eq:kmpv4} hold for an $I \ne 
\varnothing$.

Then there exists $C(\Omega, p, \kappa, \mathcal A, \mathbf U) > 0$ such that
\begin{equation}
\label{eq:kmpv-trunc-1}
\int_\Omega \sum_{i=1}^N |f_{i}|^p \, \mathrm dx
\le
C(\Omega, p, \kappa, \mathcal A, \mathbf U)
\int_\Omega \sum_{i=1}^N \tilde u_{i} (|f_{i}|^p + | \nabla f_{i} |^p) \, 
\mathrm dx
\end{equation}
for all $\mathbf u  = (u_1, \dots, u_N)\in \mathbf U$ and $(A, \mathbf m) \in 
\mathcal A$,
where
\begin{equation}
\label{eq:kmpv-trunc-2}
\tilde u_{in}(x) = \min(u_{in}(x), M)
.
\end{equation}
\end{theorem}

\section{Auxiliary functions}
\label{sec:aux}

In this section we collect a few auxiliary results concerning systems of 
affine functions on $\mathbb R^N$
\begin{equation}
\label{eq:f}
f_i(u_1, \dots, u_N) = m_i - \sum_{j = 1}^N a_{ij} u_j \quad (i = 1, \dots, N)
\end{equation}
with scalar coefficients.  The proofs of the statements can be found in the 
Appendix.

We say that the coefficients in~\eqref{eq:f} are \emph{admissible}, if they 
satisfy  Assumptions~\ref{a:bounded}, \ref{a:det-a}, and \ref{a:det-am} with a 
fixed $\kappa > 0$.

Given $ I = \{i_1, \dots, i_r\} \subset \{1, \dots, N\}$, $i_1 < \dots < i_r$, 
and $j \in \{1, \dots, N\}$, denote the determinants in the right-hand sides 
of~\eqref{eq:am3} and~\eqref{eq:am4} by
\begin{equation*}
\Delta_I
=
\begin{vmatrix}
a_{i_1i_1} & \cdots  & a_{i_1i_r} \\
\vdots   & \ddots & \vdots \\
a_{i_ri_1} & \cdots & a_{i_ri_r}
\end{vmatrix}
,
\quad
\Delta_{I,j}=
\begin{vmatrix}
a_{i_1i_1} & \cdots  & a_{i_1i_r} & m_{i_1} \\
\vdots   & \ddots & \vdots & \vdots \\
a_{i_ri_1} & \cdots & a_{i_ri_r} & m_{i_r} \\
a_{ji_1} & \cdots & a_{ji_r} & m_{j}
\end{vmatrix}
.
\end{equation*}

\begin{remark}
\label{rem:cramer}
If the determinants $\Delta_I$ are nonzero, the systems~\eqref{eq:vi} are 
exactly determined.  Denoting the solution of~\eqref{eq:vi} by $\mathbf u_I = 
(u_{I1}, \dots, u_{IN})$ as before, for any $I$ we have
\begin{equation}
\label{eq:ufu}
u_{\mathsf CIi} =
\frac{
\Delta_{I\setminus\{i\}, i}
}{
\Delta_{I}
}
,
\quad
f_j(\mathbf u_{\mathsf CI}) =
\frac{
\Delta_{I, j}
}{
\Delta_{I}
}
,
\end{equation}
where $\mathsf CI = \{1, \dots, N\} \setminus I$.  Thus, in the case 
of admissible coefficients the values of $\mathbf u_I$ and of $f_j(\mathbf 
u_I)$ are nonnegative and bounded by a constant depending only on $\kappa$, 
but not on the particular choice of coefficients.  Moreover, if $j \notin I$, 
then $u_{Ij}$ are bounded away from zero by a constant depending on $\kappa$, 
but not on the coefficients.  If $i \in I$, the same is true for 
$f_i(\mathbf u_I)$.
\end{remark}

We want to geometrically interprete the positivity of $\Delta_I$ and 
$\Delta_{I,j}$, involved in Assumptions~\ref{a:det-a} and~\ref{a:det-am}. To 
this end, consider the system of linear inequalities in $\mathbb R^N$:
\begin{equation}
\label{eq:ineqs}
\left\{
\begin{array}{l}
f_i \ge 0 \quad (i = 1, \dots, N),
\\
u_i \ge 0 \quad (i = 1, \dots, N).
\end{array}
\right.
\end{equation}

\begin{proposition}
\label{pr:cuboid}
Suppose that $\Delta_I \ne 0$ for any $I \subset \{1, \dots, N\}$.  Then 
$\Delta_I > 0$ and $\Delta_{I,j} > 0$ for any $ I \subset \{1, \dots, N\}$ and 
$j \notin I$ if and only if the solution set of~\eqref{eq:ineqs} is a polytope 
with vertices $\{\mathbf u_I \colon I \subset \{1, \dots, N\}\}$ 
combinatorially isomorphic to a cube.
\end{proposition}

One corollary of Proposition~\ref{pr:cuboid} is that in the case of 
admissible coefficients no vertex (and hence, no point whatsoever) of the 
polytope~\eqref{eq:ineqs} satisfies any of the the equations $u_i = 0 = f_i$.  
A strengthened version of this observation stated in the following lemma plays 
a crucial role in our proof.

\begin{lemma}
\label{lem:obs}
There exists $\sigma = \sigma(\kappa)$ such that if for some admissible coefficients, some index $j$, and some $\mathbf u = (u_1, \dots, u_N) \ge 0$ we have
$u_j \le \sigma$ and 
$f_j(\mathbf u) \le \sigma $, then
\begin{equation}
\label{eq:obs}
\min_{i} f_i(\mathbf u)
\le - \sigma
.
\end{equation}
\end{lemma}

Now we introduce a few auxiliary functions.  Fix $p \ge 1$ and set
\begin{gather*}
g = \sum_{i = 1}^N | f_i |^p
,
\\
v =
\begin{cases}
\frac 1g \sum_{i = 1}^N u_i | f_i |^p & \text{ if } g \ne 0
\\
\text{whatever between $\min_i u_i$ and $\max_i u_i$} & \text{ if } g = 0
.
\end{cases}
\end{gather*}
Observe that $g$ and $v$ are nonnegative on $\mathbb R^N_+$ and
\begin{gather}
g = 0 \Leftrightarrow f_i = 0 \ (i = 1, ..., N) \Leftrightarrow \mathbf u = 
\mathbf{u}_{\varnothing}
\label{eq:zerog}
,
\\
v = 0 \Leftrightarrow \mathbf u = \mathbf{u}_I \quad (I \ne \varnothing)
.
\label{eq:zerov}
\end{gather}
Also note the identity
\begin{equation}
\label{eq:vg}
vg = \sum_{i = 1}^N u_i | f_i |^p
\end{equation}
and the inequality
\begin{equation}
\label{eq:uvu}
\min_i u_i \le v \le \max_i u_i
.
\end{equation}

Developing observation~\eqref{eq:zerov}, the following lemma and its 
corollaries state that $v(\mathbf u)$ is small only in the neighbourhood of 
the set $\{\mathbf{u}_I \colon I \ne \varnothing\}$.  This allows to use the 
function $v$ to prove convergence of the form~\eqref{eq:varconvergence} and 
\eqref{eq:kmpv3}--\eqref{eq:kmpv4}.

\begin{lemma}
\label{lem:gstab}
We have
\begin{equation}
\label{eq:gstab}
\lim_{\substack{v(\mathbf u) \to 0\\
\mathbf u \in \mathbb R^N_+
}} \min_{I \ne \varnothing} | \mathbf{u} - \mathbf{u}_I |  = 0,
\end{equation}
where the limit is uniform with respect to admissible coefficients.
\end{lemma}
\begin{corollary}
\label{cor:genlimv}
There exists $\sigma > 0$ such that for any $\varepsilon > 0$ there
exists $\delta = \delta(\varepsilon, \kappa) > 0$ such that if $\mathbf u \ge 
0$ and $v(\mathbf{u}) < \delta$ for some admissible coefficients, then there 
exists $I \ne \varnothing$ such that
\begin{gather}
\sum_{i \in I} u_i + \sum_{j \notin I} |f_{j}(\mathbf{u})| < \varepsilon
,
\label{eq:genlimv1}
\\
f_{i}(\mathbf{u}) > \sigma \quad (i \in I)
.
\label{eq:genlimv2}
\end{gather}
\end{corollary}
\begin{corollary}
\label{cor:gstab}
We have
\begin{equation}
\label{eq:cor-gstab}
\lim_{\substack{v(\mathbf u) \to 0\\
\mathbf u \in \mathbb R^N_+
}} \min_{I \ne \varnothing}
\left(\sum_{i \in I} u_i + \sum_{j \notin I} |f_j| \right)
= 0,
\end{equation}
where the limit is uniform with respect to admissible coefficients.
\end{corollary}

As $v$ vanishes only at the points $\mathbf u_I$ ($I \ne \varnothing$), it 
follows from Remark~\ref{rem:cramer} that $v = 0$ implies $f_i > 0$ for some 
$i$.  The following corollary of Lemma~\ref{lem:gstab} extends this 
observation to the case of small $v$.

\begin{corollary}
\label{cor:sigma1}
There exist $\varepsilon_0 > 0$ and $\sigma > 0$ depending on $\kappa$, but 
not on admissible coefficients, such that if $v(\mathbf{u}) <
\varepsilon_0$, there exists $i$ such that $f_{i} (\mathbf{u})> \sigma$.
\end{corollary}

\section{Proof of the theorems}

\subsection{Preliminaries}
\label{sec:preliminaries}

We prove Theorem~\ref{th:kmpv-trunc}, and Theorem~\ref{th:kmpv} follows.

Assume that the theorem is not true and there exist a sequence 
$\{\varepsilon_n\}$, sequences of coefficients $\{A_n\} \subset C^1(\overline 
\Omega; M_N(\mathbb R))$ and $\{\mathbf m_n\} \subset C^1(\overline 
\Omega;\mathbb R^N)$ satisfying Assumptions~\ref{a:bounded}--\ref{a:det-am} 
with some $\kappa > 0$ and a sequence $\{(u_{1n}, \dots, u_{Nn})\} \subset 
\mathbf U$ such that $\varepsilon_n \to 0$ and
\begin{equation}
\label{eq:p01}
\int_\Omega \sum_{i = 1}^N \tilde u_{in} (|f_{in}|^p + |\nabla f_{in}|^p)
\, \mathrm dx
\le
\varepsilon_n^2
\int_\Omega
\sum_{i = 1}^N
|f_{in}|^p
\, \mathrm dx
,
\end{equation}
where $f_{in}$ corresponds to $A_n$ and $\mathbf m_n$.

We claim that without loss of generality the functions $\{\mathbf u_n\}$ can 
be assumed to be smooth.  Indeed, we always can assume
that $\mathbf U$ is open in the relative topology of the cone of nonnegative 
functions in $W^1_p$, otherwise we can replace it by a small enlargement of $\mathbf U$ without affecting the hypothesis of the theorem.  Then by Meyers-Serrin theorem we can approximate $\mathbf u_{n}$ by smooth 
functions from $\mathbf U$ such that \eqref{eq:p01} holds with 
$\varepsilon_n^2$ replaced by $2\varepsilon_n^2$.


Denote
\begin{equation*}
  \tilde v_n(x) =
  \begin{cases}
    \frac{1}{g_{n}(x)} \sum_{i = 1}^N \tilde u_{in}(x) | f_{in}(x) |^p
    & \text{ if } g_n(x) \ne 0
    \\
    v_n(x) & \text{ if } g_{n}(x)= 0
    .
  \end{cases}
\end{equation*}

It is obvious that
\begin{equation*}
  u_{in}(x) \le \varepsilon_n
  \quad
  \text{if and only if}
  \quad
  \tilde u_{in}(x) \le \varepsilon_n
  .
\end{equation*}

It follows from Lemma~\ref{lem:gstab} and~\eqref{eq:7.1} that there
exists $\delta > 0$ independent of $n$ and $x$ such that if $v_n(x) <
\delta$, then for any $i$ we have $u_{in}(x) < M$ and, consequently,
$u_{in}(x) = \tilde u_{in}(x)$ and $v_n(x) = \tilde v(x)$.  In
particular, there is no loss of generality in assuming that
\begin{equation*}
  v_n(x) \le \varepsilon_n
  \quad
  \text{if and only if}
  \quad
  \tilde v_n(x) \le \varepsilon_n
  .
\end{equation*}

Write~\eqref{eq:p01} in the form
\begin{equation*}
\int_\Omega
\tilde v_n g_n \, \mathrm dx
+
\int_\Omega \sum_{i = 1}^N \tilde u_{in} |\nabla f_{in}|^p
\, \mathrm dx
\le
\varepsilon_n^2
\int_\Omega
g_n
\, \mathrm dx
,
\end{equation*}
whence
\begin{multline*}
\int_{[v_n \le \varepsilon_n]}
\tilde v_n g_n \, \mathrm dx
+
\varepsilon_n
\int_{[v_n > \varepsilon_n]}
g_n \, \mathrm dx
+
\int_\Omega \sum_{i = 1}^N \tilde u_{in} |\nabla f_{in}|^p
\, \mathrm dx
\\ \le
\varepsilon_n^2
\left(
\int_{[v_n \le \varepsilon_n]}
g_n
\, \mathrm dx
+
\int_{[v_n > \varepsilon_n]}
g_n
\, \mathrm dx
\right)
.
\end{multline*}
Dropping a nonnegative term on the left-hand side and dividing both sides by 
$\varepsilon_n$, we obtain
\begin{equation}
\label{eq:p02}
\frac{1}{\varepsilon_n}
\int_\Omega \sum_{i = 1}^N \tilde u_{in} |\nabla f_{in}|^p
\, \mathrm dx
\le
-(1 - \varepsilon_n)
\int_{[v_n > \varepsilon_n]}
g_n \, \mathrm dx
+
\varepsilon_n
\int_{[v_n \le \varepsilon_n]}
g_n
\, \mathrm dx
.
\end{equation}
Lemma~\ref{lem:gstab} implies that if $v$ is bounded, so is $\mathbf u$, so 
there exists $M > 0$ such that $g \le M$ whenever $v < 1$.  Without loss of 
generality, $\varepsilon_n < 1$, so from~\eqref{eq:p02} we conclude
\begin{equation}
\label{eq:p03}
\frac{1}{\varepsilon_n}
\int_\Omega \sum_{i = 1}^N \tilde u_{in} |\nabla f_{in}|^p
\, \mathrm dx
\le
M \varepsilon_n
|[ v_n \le \varepsilon_n ]|
.
\end{equation}

Moreover, it follows from~\eqref{eq:p02} that the integral
\begin{equation*}
\int_{[v_n > \varepsilon_n]}
g_n \, \mathrm dx
\end{equation*}
is bounded uniformly with respect to $n$.  Hence the sequence $\{g_n\}$ is 
bounded in $L^1$ and $\{\mathbf u_n\}$ is bounded in $L^p$.  It remains to 
show that $\{\mathbf u_n\}$ satisfies \eqref{eq:kmpv3} and \eqref{eq:kmpv4} 
for some nonempty $I$ in order to obtain a contradiction.

\begin{lemma}
\label{lem:claim1}
Given $a > 0$,
\begin{equation}
\label{eq:claim1}
\lim_{n \to \infty}
|
[ v_n > \varepsilon_n ]
\cap
[ g_n > a ]
|
=0
.
\end{equation}
\end{lemma}
\begin{proof}
We have:
\begin{equation}
\label{eq:claim101}
|
[ v_n > \varepsilon_n ]
\cap
[ g_n > a ]
|
\le
\frac 1a
\int_{[v_n > \varepsilon_n]\cap[g_n > a]}
g_n
\, \mathrm dx
\le
\frac 1a
\int_{[v_n > \varepsilon_n]}
g_n
\, \mathrm dx
.
\end{equation}
Inequality~\eqref{eq:p02} implies
\begin{equation*}
-(1 - \varepsilon_n)
\int_{[v_n > \varepsilon_n]}
g_n \, \mathrm dx
+
\varepsilon_n
\int_{[v_n \le \varepsilon_n]}
g_n
\, \mathrm dx
\ge 0
,
\end{equation*}
so we can estimate the last integral in~\eqref{eq:claim101} and 
obtain
\begin{equation*}
|
[ v_n > \varepsilon_n ]
\cap
[ g_n > a ]
|
\le
\frac{\varepsilon_n}{a(1 - \varepsilon_n)}
\int_{[v_n \le \varepsilon_n]}
g_n
\, \mathrm dx
\le
\frac{M\varepsilon_n |[ v_n \le \varepsilon_n ]|}{a(1 - \varepsilon_n)}
\to 0
\quad (n \to \infty)
\end{equation*}
and \eqref{eq:claim1} is proved.
\end{proof}

\begin{lemma}
\label{lem:claim2}
Given $a$, there exists $C_a$ such that for large $n$,
\begin{equation}
\label{eq:claim2}
| [ g_n > a  ] |
\le
C_a
| [ v_n \le \varepsilon_n  ] |
.
\end{equation}
\end{lemma}
\begin{proof}
Using the estimate
\begin{equation*}
|
[ v_n > \varepsilon_n ]
\cap
[ g_n > a ]
|
\le
\frac{M\varepsilon_n |[ v_n \le \varepsilon_n ]|}{a(1 - \varepsilon_n)}
\end{equation*}
obtained in the proof of Lemma~\ref{lem:claim1}, we get
\begin{equation*}
| [ g_n > a  ] |
\le
| [ v_n \le \varepsilon_n  ] |
+
| [ v_n > \varepsilon_n  ] \cap [ g_n > a  ] |
\le
\left(
1 +
\frac{M\varepsilon_n}{a(1-\varepsilon_n)}
\right)
| [ v_n \le \varepsilon_n  ] |
,
\end{equation*}
and the lemma follows.
\end{proof}

\subsection{Limit behaviour of the sequences}
\label{sec:limit-behav-sequ}

Now we are ready to consider the dynamics of $\mathbf{u}_n$ in
detail.

We choose and fix $\varepsilon_0 \in (0, 1)$ and $\sigma > 0$ satisfying 
Corollary~\ref{cor:sigma1} and Lemma~\ref{lem:obs}.  As those numbers do not 
depend on admissible coefficients, they satisfy

\begin{condition}
\label{a:sigma1}
If $v_n(x) < \varepsilon_0$, there exists $i$ such that
$|f_{in}(x)| > \sigma$.
\end{condition}

\begin{condition}
\label{a:sigma2}
If $f_{in}(x) \le \sigma$ and $u_{in} (x) \le \sigma$, then there exists $j 
\ne i$ such that $f_j \le  -\sigma$.
\end{condition}

Given $I \subset \{1, \dots, N\}$, define
\begin{equation*}
  A_n(I)
= \left(\bigcap_{i \in I} [f_{in} > \sigma]\right)
\cap
\left(\bigcap_{j \notin I} \left[|f_{jn}| < \frac{\sigma}{2} \right]\right)
.
\end{equation*}

\begin{lemma}
  \label{lem:quasipartition}
  We have
  \begin{equation}
    \label{eq:quasipartition}
    \lim_{n \to \infty} \sum_I |A_n(I)| = |\Omega|
.
  \end{equation}
\end{lemma}
\begin{proof}
We prove the lemma by showing the inclusion
\begin{equation}
  \label{eq:quasipartition0}
  [v_n \le \varepsilon_n] \cup [ g_n \le a] \subset \bigcup_I A_n(I)
\end{equation}
with suitable $a$ and evoking Lemma~\ref{lem:claim1}.

Take $a = (\sigma/2)^p$, then the inequality $g_n \le a$ clearly
implies $|f_i| \le \sigma/2$ for any $i$, so
\begin{equation}
  \label{eq:quasipartition1}
  [g_n \le a]
  \subset \bigcap_{i=1}^N \left[
    |f_{in}| \le \frac{\sigma}{2}
  \right]
  = A_n(\varnothing)
  .
\end{equation}

Now suppose that $v_n(x) \le \varepsilon_n$ for some $x$.  Applying
Corollary~\ref{cor:genlimv} with $\varepsilon = \sigma/2$, we find $I
\ne \varnothing$ such that
\begin{equation*}
f_{in}(x) > \sigma \quad (i \in I)
\end{equation*}
and
\begin{equation*}
  \sum_{j \notin I} |f_{jn} (x)| < \frac{\sigma}{2}
\end{equation*}
whenever $\varepsilon_n < \delta$ for some $\delta > 0$ independent of
$n$ and $x$.  Consequently,
\begin{equation*}
  x \in A_n(I)
\end{equation*}
and we have the inclusion
\begin{equation}
  \label{eq:quasipartition2}
  [ v_n \le \varepsilon_n ] \subset \bigcup_{I \ne \varnothing} A_n(I)
  .
\end{equation}
Combining \eqref{eq:quasipartition1} and \eqref{eq:quasipartition2}, we obtain 
\eqref{eq:quasipartition0}.

By Lemma~\ref{lem:claim1}, the measure of the left-hand side of
inclusion \eqref{eq:quasipartition0} converges to $|\Omega|$, and
\eqref{eq:quasipartition} follows.
\end{proof}

We can assume that the limits
\begin{equation*}
  \lim_{k \to \infty} A_{n} (I)
\end{equation*} exist.  In view of Lemma~\ref{lem:quasipartition} we
face three logical possibilities:

(i) $\lim_{n \to \infty} A_{n} (I) = |\Omega|$ for some $I \ne
\varnothing$;

(ii) $\lim_{n \to \infty} A_{n} (\varnothing) = |\Omega|$;

(iii) $\lim_{n \to \infty} A_{n} (I_s) > 0$ ($s = 1, 2$) with $I_1
\ne I_2$.

We conclude the proof of Theorem~\ref{th:kmpv-trunc} by examining the
alternatives (i)--(iii).  It is fairly straightforward to demonstrate
that (i) implies \eqref{eq:kmpv3} and \eqref{eq:kmpv4}.  A more subtle
analysis based on the coarea formula and the relative isoperimetric inequality 
shows that (ii) and (iii) are in fact impossible.

Recall that the relative perimeter of a Lebesgue measurable set $A$ of (locally) finite perimeter is defined as
\begin{equation*}
P(A;\Omega) = \mu_A(\Omega)
,
\end{equation*}
where $\mu_A$ is the total variation of the Gauss--Green measure of $A$ (see 
\cite{Mag12}).  We need the following properties of the perimeters:
\begin{lemma}[\cite{Mag12}, Proposition~12.19 and Lemma~12.22]
\label{lem:maggi}
If $A$ is a set of locally finite perimeter in $\mathbb R^d$, then
\begin{equation*}
\supp \mu_A \subset \partial A;
\end{equation*}
if $A$ and $B$ are sets of (locally) finite perimeter in $\mathbb R^d$, then 
$A \cup B$ is a set of (locally) finite perimeter in $\mathbb R^d$, and, for 
$\Omega \subset \mathbb R^d$ open,
\begin{equation*}
P(A \cup B;\Omega) \le P(A; \Omega) + P(B; \Omega)
.
\end{equation*}
\end{lemma}

\subsection{Convergence in case (i)}
\label{sec:convergence-case-i}

Assume that (i) holds.  We claim that \eqref{eq:kmpv3} and
\eqref{eq:kmpv4} are valid, i.~e. for any $\varepsilon > 0$,
\begin{equation}
  \label{eq:poss0}
  \lim_{k \to \infty}
  \left|
    \left[
      \sum_{j \in I} u_{in}
      + \sum_{i \notin I} | f_{in} |
      \ge \varepsilon
    \right]
  \right|
  = 0
  .
\end{equation}

By assumption, $I \ne \varnothing$, so for any $x \in A_{n}(I)$ we
have $f_{in}(x) > \sigma$ for at least one $i$ and thus $g_{n}(x)
> \sigma^p$.  It follows from Lemma~\ref{lem:claim1} that
\begin{equation*}
  \lim_{k \to \infty} | A_{n}(I) \cap [v_{n} > \varepsilon_{n} ]
  | = 0
\end{equation*}
and consequently
\begin{equation}
  \label{eq:poss1}
  \lim_{k \to \infty} | A_{n}(I) \cap [v_{n} \le \varepsilon_{n} ]
  | = |\Omega|
  .
\end{equation}

Take $x \in A_{n}(I) \cap [v_{n} \le \varepsilon_{n}]$.  By
Corollary~\ref{cor:genlimv}, there exists $\delta > 0$ independent of
$k$ and $x$ such that for some $I_{k, x}$ we have
\begin{gather}
  \sum_{j \in I_{k,x}} u_{in}(x)
  + \sum_{i \notin I_{k,x}} | f_{in}(x) |
  < \varepsilon_n
  ,
  \label{eq:poss2}
  \\
  f_{jn}(x) > \sigma \quad (j \in I_{k, x})
  ,
  \label{eq:poss3}
\end{gather}
provided that $v_{n}(x) < \delta$, which holds without loss of
generality.  However, \eqref{eq:poss2} and \eqref{eq:poss3} are only
compatible with the definition of $A_{n}(I)$ in the case $I_{k, x} =
I$.  Thus,
\begin{equation*}
  \sum_{j \in I} u_{in}(x)
  + \sum_{i \notin I} | f_{in}(x) |
  < \varepsilon
  \quad (x \in A_{n}(I) \cap [v_{n} \le \varepsilon_{n}])
  ,
\end{equation*}
or, equivalently,
\begin{equation*}
  A_{n}(I) \cap [v_{n} \le \varepsilon_{n}]
  \subset
  \left[
    \sum_{j \in I} u_{in}
    + \sum_{i \notin I} | f_{in} |
    < \varepsilon
  \right]
  .
\end{equation*}
Consequently,
\begin{equation*}
  \left|
    \left[
      \sum_{j \in I} u_{in}
      + \sum_{i \notin I} | f_{in} |
      \ge \varepsilon
    \right]
  \right|
  \le
  |\Omega| - |A_{n}(I) \cap [v_{n} \le \varepsilon_{n}]|
\end{equation*}
and by \eqref{eq:poss1}, the limit~\eqref{eq:poss0} holds.

\subsection{Impossibility of case (ii)}
\label{sec:imposs-case-ii}

We argue by contradiction that case (ii) is impossible.  Thus, we assume that
\begin{equation}
  \label{eq:case1}
  \lim_{n \to \infty}
  \left|
    \bigcap_{i=1}^N \left[ |f_{in}| \le \frac{\sigma}{2} \right]
  \right|
  = | \Omega|
  .
\end{equation}

Given $t \in (\sigma/2, \sigma)$ define the set
\begin{equation*}
A_n(t) = \bigcup_{i=1}^N [|f_{in}| > t]
.
\end{equation*}
For fixed $n$, $A_n(t)$ decreases with respect to $t$.  We need to establish 
several properties of these sets.
\begin{lemma}
\label{lem:c1a1}
For fixed $t$,
\begin{equation*}
\lim_{n \to \infty} |A_n(t)| = 0
.
\end{equation*}
\end{lemma}
\begin{proof}
We have:
\begin{equation*}
A_n(t)
\subset \bigcup_{i=1}^N \left[ |f_{in}| > \frac{\sigma}{2} \right]
= \Omega \setminus \bigcap_{i=1}^N \left[ |f_{in}| \le \frac{\sigma}{2} 
\right]
,
\end{equation*}
so
\begin{equation*}
|A_n(t)|
\le
|\Omega| -
\left| \bigcap_{i=1}^N [ |f_{in}| \le \sigma ] \right|
\to 0
\quad (n \to \infty)
\end{equation*}
according to~\eqref{eq:case1}.
\end{proof}

\begin{lemma}
\label{lem:c1a2}
The following inclusion holds:
\begin{equation}
\label{eq:c1a2}
\partial_\Omega A_n(t)
\subset
\bigcap_{i = 1}^N [ u_{in} > \varepsilon_n ]
.
\end{equation}
\end{lemma}
\begin{proof}
We take an $x \in \Omega$ such that $u_{jn}(x) \le \varepsilon_n$ for some $j$ 
and show that $x$ is an interior point of $A_n(t)$.  There are two 
possibilities: either $f_{jn}(x) > t$ or $f_{jn}(x) \le t$.  In the former 
case we see immediately that $x$ is an interior point of $A_n(t)$.  In the 
latter case we have $f_{jn}(x) \le \sigma$, and applying 
Condition~\ref{a:sigma2} we find $f_{in}$ such that $f_{in}(x) \le -\sigma < 
- t$.  But then $|f_{in}(x)| > t$, and again $x$ is an interior point of 
$A_n(t)$.
\end{proof}

\begin{lemma}
\label{lem:c1a3}
If $n$ is sufficiently large, the following inclusion holds:
\begin{equation}
\label{eq:c1a3}
A_n(t) \supset [ v_n \le \varepsilon_n ]
.
\end{equation}
\end{lemma}
\begin{proof}
By Condition~\ref{a:sigma1}, for large $n$ we have
\begin{equation*}
[v_n \le \varepsilon_n] \subset
\bigcup_{i =1}^N [ |f_{in}| > \sigma ]
\subset A_n(t)
.
\end{equation*}
\end{proof}

It follows from Lemma~\ref{lem:c1a1} that we can write the isoperimetric 
inequality
\begin{equation}
\label{eq:p10}
P(A_n(t); \Omega) \ge c_\Omega |A_n(t)|^\frac{d-1}{d}
.
\end{equation}

Estimate the left-hand side of~\eqref{eq:p03}:
\begin{align*}
\frac{1}{\varepsilon_n}
\int_\Omega \sum_{i = 1}^N \tilde u_{in} |\nabla f_{in}|^p
\, \mathrm dx
& \ge
\frac{1}{\varepsilon_n}
\sum_{i = 1}^N
\int_{\left[ g_n > \left(\frac{\sigma}{2}\right)^p \right]} \tilde u_{in} 
|\nabla f_{in}|^p
\, \mathrm dx
\\
& \ge
\frac{1}{\varepsilon_n \left| \left[ g_n > \left(\frac{\sigma}{2}\right)^p 
\right] \right|^{p-1}}
\sum_{i = 1}^N
\left(
\int_{\left[ g_n > \left(\frac{\sigma}{2}\right)^p \right]} \tilde 
u_{in}^{1/p} |\nabla f_{in}|
\, \mathrm dx
\right)^p
\\
& \ge
\frac{1}{\varepsilon_n N^{p-1}
\left| \left[ g_n > \left(\frac{\sigma}{2}\right)^p \right] \right|^{p-1}}
\left(
\sum_{i = 1}^N
\int_{\left[g_n > \left(\frac{\sigma}{2}\right)^p\right]} \tilde u_{in}^{1/p} 
|\nabla f_{in}|
\, \mathrm dx
\right)^p
\\
& \ge
\frac{1}{N^{p-1}
\left| \left[ g_n > \left(\frac{\sigma}{2}\right)^p \right] \right|^{p-1}}
\left(
\sum_{i = 1}^N
\int_{\left[g_n > \left(\frac{\sigma}{2}\right)^p\right] \cap [ u_{in} > 
\varepsilon_n]} |\nabla f_{in}|
\, \mathrm dx
\right)^p
.
\end{align*}
Apply the coarea formula~\cite[Theorem~13.1, formula (13.10)]{Mag12} and 
Lemma~\ref{lem:maggi}:
\begin{multline*}
\frac{1}{\varepsilon_n}
\int_\Omega \sum_{i = 1}^N \tilde u_{in} |\nabla f_{in}|^p
\, \mathrm dx
\\
\ge
\frac{1}{N^{p-1} \left| \left[ g_n > \left(\frac{\sigma}{2}\right)^p \right] 
\right|^{p-1}}
\Bigg(
\sum_{i = 1}^N
\int_{0}^{\infty}
P \left([f_{in} > t] ; \left[g_n > \left(\frac{\sigma}{2}\right)^p\right] \cap 
[u_{in} > \varepsilon_n] \right) \, \mathrm dt
\\
+
\int_{-\infty}^{0}
P \left([f_{in} < t] ; \left[g_n > \left(\frac{\sigma}{2}\right)^p\right] \cap 
[u_{in} > \varepsilon_n] \right) \, \mathrm dt
\Bigg)^p
\\
\ge
\frac{1}{N^{p-1} \left| \left[ g_n > \left(\frac{\sigma}{2}\right)^p \right] 
\right|^{p-1}}
\Bigg(
\sum_{i = 1}^N
\int_{\sigma/2}^{\sigma}
\Bigg(
P \left([f_{in} > t] ; \left[g_n > \left(\frac{\sigma}{2}\right)^p\right] \cap 
[u_{in} > \varepsilon_n] \right)
\\
+
P \left([f_{in} < -t] ; \left[g_n > \left(\frac{\sigma}{2}\right)^p\right] 
\cap [u_{in} > \varepsilon_n] \right)
\Bigg)
\, \mathrm dt
\Bigg)^p
\\
\ge
\frac{1}{N^{p-1} \left| \left[ g_n > \left(\frac{\sigma}{2}\right)^p \right] 
\right|^{p-1}}
\left(
\sum_{i = 1}^N
\int_{\sigma/2}^{\sigma}
P \left([|f_{in}| > t] ; \left[g_n > \left(\frac{\sigma}{2}\right)^p\right] 
\cap [u_{in} > \varepsilon_n] \right) \, \mathrm dt
\right)^p
\\
=
\frac{1}{N^{p-1} \left| \left[ g_n > \left(\frac{\sigma}{2}\right)^p \right] 
\right|^{p-1}}
\left(
\sum_{i = 1}^N
\int_{\sigma/2}^{\sigma}
\mu_{[|f_{in}| > t]} \left(\left[g_n > \left(\frac{\sigma}{2}\right)^p\right] 
\cap [u_{in} > \varepsilon_n] \right) \, \mathrm dt
\right)^p
\end{multline*}
Observe that for $t \in (\sigma/2, \sigma)$ by Lemma~\ref{lem:maggi} we have
\begin{equation*}
\supp \mu_{[|f_{in}| > t]}
\subset \partial[|f_{in}| > t]
\subset[|f_{in}| = t]
\subset \left[g_n > \left(\frac\sigma2\right)^p\right]
,
\end{equation*}
so we can proceed as follows:
\begin{multline*}
\frac{1}{\varepsilon_n}
\int_\Omega \sum_{i = 1}^N \tilde u_{in} |\nabla f_{in}|^p
\, \mathrm dx
\\
\ge
\frac{1}{N^{p-1} \left| \left[ g_n > \left(\frac{\sigma}{2}\right)^p \right] 
\right|^{p-1}}
\left(
\sum_{i = 1}^N
\int_{\sigma/2}^{\sigma}
\mu_{[|f_{in}| > t]} ([u_{in} > \varepsilon_n] ) \, \mathrm dt
\right)^p
\\
\ge
\frac{1}{N^{p-1} \left| \left[ g_n > \left(\frac{\sigma}{2}\right)^p \right] 
\right|^{p-1}}
\left(
\sum_{i = 1}^N
\int_{\sigma/2}^{\sigma}
\mu_{[|f_{in}| > t]} \left(\bigcap_{i=1}^N [u_{in} > \varepsilon_n] \right) \, 
\mathrm dt
\right)^p
\\
=
\frac{1}{N^{p-1} \left| \left[ g_n > \left(\frac{\sigma}{2}\right)^p \right] 
\right|^{p-1}}
\left(
\sum_{i = 1}^N
\int_{\sigma/2}^{\sigma}
P\left([|f_{in}| > t]; \left(\bigcap_{i=1}^N [u_{in} > \varepsilon_n] \right) 
\right) \, \mathrm dt
\right)^p
\\
\ge
\frac{1}{N^{p-1} \left| \left[ g_n > \left(\frac{\sigma}{2}\right)^p \right] 
\right|^{p-1}}
\left(
\int_{\sigma/2}^{\sigma}
P\left(A_n(t); \left(\bigcap_{i=1}^N [u_{in} > \varepsilon_n] \right) \right) 
\, \mathrm dt
\right)^p
.
\end{multline*}
Now~\eqref{eq:c1a2} implies that
\begin{equation*}
\supp \mu_{A_n(t)} \cap \Omega \subset \bigcap_{i=1}^N [u_{in} > 
\varepsilon_n]
,
\end{equation*}
so we can write
\begin{equation*}
\frac{1}{\varepsilon_n}
\int_\Omega \sum_{i = 1}^N \tilde u_{in} |\nabla f_{in}|^p
\, \mathrm dx
\ge
\frac{1}{N^{p-1} \left| \left[ g_n > \left(\frac{\sigma}{2}\right)^p \right] 
\right|^{p-1}}
\left(
\int_{\sigma/2}^{\sigma}
P(A_n(t); \Omega ) \, \mathrm dt
\right)^p
.
\end{equation*}
Employing the isoperimetric inequality~\eqref{eq:p10} to get
\begin{equation*}
\frac{1}{\varepsilon_n}
\int_\Omega \sum_{i = 1}^N \tilde u_{in} |\nabla f_{in}|^p
\, \mathrm dx
\ge
\frac{c_\Omega^p}{N^{p-1} \left| \left[ g_n > \left(\frac{\sigma}{2}\right)^p 
\right] \right|^{p-1}}
\left(
\int_{\sigma/2}^{\sigma}
|A_n(t)|^{\frac{d-1}{d}}
\, \mathrm dt
\right)^p
.
\end{equation*}
Estimate $|A_n(t)|$ using the inclusion~\eqref{eq:c1a3} and 
Lemma~\ref{lem:claim2}:
\begin{align*}
\frac{1}{\varepsilon_n}
\int_\Omega \sum_{i = 1}^N \tilde u_{in} |\nabla f_{in}|^p
\, \mathrm dx
& \ge
\frac{c_\Omega^p (\sigma/2)^p |[v_n \le \varepsilon_n]|^{p(1 - 
1/d)}}{C_{(\sigma/2)^p}N^{p-1} | [ v_n \le \varepsilon_n ] |^{p-1}}
\\
& =
\frac{
c_\Omega^p (\sigma/2)^p
}{
C_{(\sigma/2)^p}N^{p-1}
}
| [ v_n \le \varepsilon_n ] |^{1 - \frac{p}{d}}
.
\end{align*}
Combining obtained estimate with~\eqref{eq:p03}, we get:
\begin{equation*}
\frac{
c_\Omega^p (\sigma/2)^p
}{
C_{(\sigma/2)^p}N^{p-1}
}
| [ v_n \le \varepsilon_n ] |^{1 - \frac{p}{d}}
\le
M \varepsilon_n
|[v_n \le \varepsilon_n]|
,
\end{equation*}
whence
\begin{equation*}
\frac{
c_\Omega^p (\sigma/2)^p
}{
C_{(\sigma/2)^p}N^{p-1}
}
\le
M \varepsilon_n
|[v_n \le \varepsilon_n]|^{\frac pd} \to 0
\quad (n \to \infty)
,
\end{equation*}
contrary to the fact that the left-hand side is a positive constant 
independent of $n$.

This contradiction means that at least assumption~\eqref{eq:case1} is 
impossible.

\subsection{Impossibility of case (iii)}
\label{sec:imposs-case-iii}

We complete the proof of Theorem~\ref{th:kmpv} by demonstrating that the case 
(iii) is also impossible.  We argue by contradiction. We assume that there
exist $I_1 \ne I_2$ such that
\begin{equation}
  \label{eq:1}
  |A_n(I_s)| \ge \mu_s > 0 \quad (s = 1, 2)
  .
\end{equation}
Without loss of generality, $1 \in I_1 \setminus I_2$.

Given $t \in (\sigma/2, \sigma)$, define the set
\begin{equation}
  \label{eq:3}
A_n(t)
=
[ f_{1n} > t ]
\cup
\left(
\bigcup_{i=1}^{N}
[ f_{in} < -t ]
\right)
.
\end{equation}
If $n$ is fixed, the sets $A_n(t)$ decrease with respect to $t$.

\begin{lemma}
\label{lem:c2a1}
The relative perimeter of $A_n(t)$ can be estimated as
\begin{equation}
\label{eq:c2a1}
P(A_n(t); \Omega) \ge p_0
,
\end{equation}
where $p_0 > 0$ is independent of $t$ and $n$.
\end{lemma}
\begin{proof}
First of all, observe the inclusions
\begin{equation}
  \label{eq:2}
  A_n(I_1) \subset A_n(t) \subset \Omega \setminus A_n(I_2)
  .
\end{equation}
Indeed, if $x \in A_n(I_1)$, then $f_{1n}(x) > \sigma$, so $x$ belongs
to the first set in the right-hand side of \eqref{eq:3}, and the first
inclusion in \eqref{eq:2} holds.  On the other hand, if $x \in
A_n(t)$, then either $f_{1n}(x) > t > \sigma/2$ or $f_{in}(x) < -t <
-\sigma/2$ for some $i$.  As $1 \notin I_2$, it is clear that in both
cases $x \notin A_n(I_2)$, so the second inclusion in \eqref{eq:2} is
also valid.

The isoperimetric inequality for $A_n(t)$ reads
\begin{equation*}
  P(A_n(t); \Omega) \ge
  c_\Omega
  \Big(
  \min(|A_n(t)|, |\Omega \setminus A_n(t)|)
  \Big)^\frac{d-1}{d}
\end{equation*}
Estimating by means of \eqref{eq:2}, we have:
\begin{equation*}
  P(A_n(t); \Omega) \ge
  c_\Omega
  (
  \min(\mu_1, 1 - \mu_2)
  )^\frac{d-1}{d}
\end{equation*}
and~\eqref{eq:c2a1} follows.
\end{proof}

\begin{lemma}
  \label{lem:1}
  The following inclusions hold:
\begin{gather}
\partial_\Omega A_n(t) \cap \partial [f_{1n} > t]  \subset [u_{1n} > 
\varepsilon_n]
\label{eq:11}
\\
\partial_\Omega A_n(t) \cap \partial [f_{in} < - t] \subset [u_{in} > 
\varepsilon_n]
\label{eq:12}
\end{gather}
\end{lemma}
\begin{proof}
If $u_{1n}(x) \le \varepsilon_n$ and $x \in \partial [f_{1n} > t]$, then 
$f_{1n}(x) = t \le \sigma$.  If $u_{in}(x) \le \varepsilon_n$ and $x \in 
\partial [f_{in} < -t]$, then $f_{in}(x) = - t \le \sigma$.  In both cases by 
Condition~\ref{a:sigma2} there exists $j$ such that $f_{jn}(x) \le - \sigma < 
-t$, so $x$ belongs to the interior of $A_n(t)$ and the lemma follows.
\end{proof}

Estimate the left-hand side of~\eqref{eq:p03}:
\begin{align*}
  \frac{1}{\varepsilon_n}
  \int_\Omega \sum_{i = 1}^N \tilde u_{in} |\nabla f_{in}|^p
  \, \mathrm dx & \ge
                  \sum_{i = 1}^N \int_{[u_{in} > \varepsilon_n]} |\nabla 
                  f_{in}|^p
                  \, \mathrm dx
  \\
                & \ge
                  \sum_{i = 1}^N
                  \frac{1}{|[u_{in} > \varepsilon_n]|^{p - 1}}
                  \left(
                  \int_{[u_{in} > \varepsilon_n]} |\nabla f_{in}|
                  \, \mathrm dx
                  \right)^p
  \\
                & \ge
                  \frac{1}{N^{p-1}|\Omega|^{p - 1}}
                  \left(
                  \sum_{i = 1}^N
                  \int_{[u_{in} > \varepsilon_n]} |\nabla f_{in}|
                  \, \mathrm dx
                  \right)^p
                  .
\end{align*}
Now we apply the coarea formula:
\begin{multline*}
  \frac{1}{\varepsilon_n}
  \int_\Omega \sum_{i = 1}^N \tilde u_{in} |\nabla f_{in}|^p
  \, \mathrm dx
  \\
  \ge 
  \frac{1}{N^{p-1} |\Omega|^{p - 1}}
  \left(
    \sum_{i = 1}^N
    \int_{0}^{\infty}
    P ([f_{in} > t] ; [u_{in} > \varepsilon_n])
    \, \mathrm dt
    +
    \sum_{i = 1}^N
    \int_{-\infty}^{0}
    P ([f_{in} < t] ; [u_{in} > \varepsilon_n])
    \, \mathrm dt
  \right)^p
  \\
  \ge
  \frac{1}{N^{p-1}|\Omega|^{p - 1}}
  \left(
    \int_{\frac\sigma2}^{\sigma}
    \left(
    P ([f_{1n} > t] ; [u_{1n} > \varepsilon_n])
  +
    \sum_{i = 1}^N
    P ([f_{in} < -t] ; [u_{in} > \varepsilon_n])
    \right)
    \mathrm dt
  \right)^p
  \\
  =
  \frac{1}{N^{p-1}|\Omega|^{p - 1}}
  \left(
    \int_{\frac\sigma2}^{\sigma}
    \left(
    \mu_{[f_{1n} > t]} ([u_{1n} > \varepsilon_n])
  +
    \sum_{i = 1}^N
    \mu_{[f_{in} < -t]}([u_{in} > \varepsilon_n])
    \right)
    \mathrm dt
  \right)^p
  .
\end{multline*}
Using~\eqref{eq:11} and \eqref{eq:12}, we obtain
\begin{multline*}
  \frac{1}{\varepsilon_n}
  \int_\Omega \sum_{i = 1}^N \tilde u_{in} |\nabla f_{in}|^p
  \, \mathrm dx
  \ge
  \frac{1}{N^{p-1}|\Omega|^{p - 1}}
  \\ \times\left(
    \int_{\frac\sigma2}^{\sigma}
    \left(
    \mu_{[f_{1n} > t]} (\partial_\Omega A_n(t) \cap \partial [f_{1n} > t])
  +
    \sum_{i = 1}^N
    \mu_{[f_{in} < -t]}(\partial_\Omega A_n(t) \cap \partial [f_{in} <
    -t ])
    \right)
    \mathrm dt
  \right)^p
  .
\end{multline*}
Using the inclusions
\begin{gather*}
\supp \mu_{[f_{1n} > t]} \subset \partial [f_{1n} > t]
,
\\
\supp \mu_{[f_{in} < t]} \subset \partial [f_{in} < -t]
,
\\
\supp\mu_{A_n(t)} \cap \Omega \subset \partial_\Omega A_n(t)
,
\end{gather*}
we get
\begin{multline*}
  \frac{1}{\varepsilon_n}
  \int_\Omega \sum_{i = 1}^N \tilde u_{in} |\nabla f_{in}|^p
  \, \mathrm dx
  \\
  \ge
  \frac{1}{N^{p-1}|\Omega|^{p - 1}}
  \left(
    \int_{\frac\sigma2}^{\sigma}
    \left(
    \mu_{[f_{1n} > t]} (\partial_\Omega A_n(t))
  +
    \sum_{i = 1}^N
    \mu_{[f_{in} < -t]}(\partial_\Omega A_n(t))
    \right)
    \mathrm dt
  \right)^p
  \\
  =
  \frac{1}{N^{p-1}|\Omega|^{p - 1}}
  \left(
    \int_{\frac\sigma2}^{\sigma}
    \left(
    P([f_{1n} > t]; \partial_\Omega A_n(t))
  +
    \sum_{i = 1}^N
    P([f_{in} < -t];\partial_\Omega A_n(t))
    \right)
    \mathrm dt
  \right)^p
  \\
  \ge
  \frac{1}{N^{p-1}|\Omega|^{p - 1}}
  \left(
    \int_{\frac\sigma2}^{\sigma}
    P(A_n(t); \partial_\Omega A_n(t))
    \mathrm dt
  \right)^p
  \\
  =
  \frac{1}{N^{p-1}|\Omega|^{p - 1}}
  \left(
    \int_{\frac\sigma2}^{\sigma}
    P(A_n(t); \Omega)
    \mathrm dt
  \right)^p
  .
\end{multline*}
Estimating the relative perimeter by Lemma~\ref{lem:c2a1}, we
conclude that
\begin{equation*}
  \frac{1}{\varepsilon_n}
  \int_\Omega \sum_{i = 1}^N \tilde u_{in} |\nabla f_{in}|^p
  \, \mathrm dx
    \ge \frac{1}{N^{p-1}|\Omega|^{p - 1}}
  \left(
    \frac{\sigma p_0}{2}
  \right)^p
  .
\end{equation*}

Now from \eqref{eq:p03} we get
\begin{equation*}
  \frac{1}{N^{p-1}|\Omega|^{p - 1}}
  \left(
    \frac{\sigma p_0}{2}
  \right)^p
  \le
  M \varepsilon_n |[v_n \le \varepsilon_n]|
  ,
\end{equation*}
where the left-hand side is a positive constant, and the right-hand
side goes to 0 as $n \to \infty$, a contradiction.

\section*{Appendix}

\begin{proof}[Proof of Proposition~\ref{pr:cuboid}]
Denote the solution set of~\eqref{eq:ineqs} by $P$.

\emph{Step 1.  Basis of the induction.}  We use induction over $N$.  A direct 
verification shows that the statement is true for $N = 1$.

\emph{Step 2.  Positivity of the determinants $\Delta_I$ implies that $P$ is 
bounded.}  It is well-known that the polyhedron $P$ is bounded if and only if the homogeneous system
\begin{equation}
\label{eq:3s10}
\left\{
\begin{array}{r}
a_{11} u_1 + \dots + a_{1N} u_N \le 0, \\
\dots \\
a_{N1} u_1 + \dots + a_{NN} u_N \le 0, \\
 u_1 \ge 0, \\
\dots \\
u_N \ge 0, \\
\end{array}
\right.
\end{equation}
admits only the trivial solution.  Assume that contrary to the claim, 
system~\eqref{eq:3s10} has a nontrivial solution.  As the system has rank $N$, 
it has a nontrivial solution $\mathbf b = (b_1, \dots, b_N)$ satisfying exactly $N - 
1$ independent inequalities with equality.  If
\begin{equation*}
b_N = 0 = a_{N1}b_1 + \dots + a_{NN}b_N
,
\end{equation*}
then $(b_1, \dots, b_{N-1})$ is a nontrivial solution of
\begin{equation*}
\left\{
\begin{array}{r}
a_{11} u_1 + \dots + a_{1,N-1} u_{N-1} \le 0, \\
\dots \\
a_{N-1,1} u_1 + \dots + a_{N-1,N-1} u_{N-1} \le 0, \\
 u_1 \ge 0, \\
\dots \\
u_{N-1} \ge 0, \\
\end{array}
\right.
,
\end{equation*}
which contradicts the induction assumption.  More generally, for no $i$ can we 
have
\begin{equation*}
b_i = 0 = a_{i1}b_1 + \dots + a_{iN}b_N
.
\end{equation*}
Thus, without loss of generality we can assume that $b_{r+1} = \dots = b_N = 
0$, $b_r = 1$, while $b_1$, \dots, $b_{r-1}$ are positive, solve
\begin{equation}
\label{eq:3s11}
\left\{
\begin{array}{l}
a_{11} u_1 + \dots + a_{1,r-1} u_{r-1} = - a_{1r}, \\
\dots \\
a_{r-1,1} u_1 + \dots + a_{r-1,r-1} u_{N-1} = -a_{r-1, r}, \\
\end{array}
\right.
\end{equation}
and satisfy
\begin{equation}
\label{eq:3s12}
a_{r1} b_1 + \dots + a_{r,r-1} b_{N-1} + a_{rr} < 0.
\end{equation}
Using Cramer's rule to solve~\eqref{eq:3s11}, we see that
\begin{equation}
\label{eq:3s13}
b_r = \frac{\Delta_{i}}{\Delta_r} \quad (i = 1, \dots, r-1),
\end{equation}
where $\Delta_i$ is the $r,i$-cofactor of the matrix
\begin{equation*}
\begin{bmatrix}
a_{11} & \dots & a_{1r} \\
\vdots & \ddots & \vdots \\
a_{r1} & \dots & a_{rr} \\
\end{bmatrix}
.
\end{equation*}
for $i = 1, \dots, r$.  Plugging the representation~\eqref{eq:3s13} 
into~\eqref{eq:3s12} and applying the Laplace formula, we obtain
\begin{equation*}
\frac{
\begin{vmatrix}
a_{11} & \dots & a_{1r} \\
\vdots & \ddots & \vdots \\
a_{r1} & \dots & a_{rr} \\
\end{vmatrix}
}{
\begin{vmatrix}
a_{11} & \dots & a_{1, r-1} \\
\vdots & \ddots & \vdots \\
a_{r-1,1} & \dots & a_{r-1, r-1} \\
\end{vmatrix}
}
< 0
,
\end{equation*}
which contradicts the positivity of all $\Delta_I$.  Thus, $P$ is bounded.

\emph{Step 3.  Positivity of the determinants implies that the set of vertices 
of $P$ is $\{\mathbf u_I \colon I \subset \{1, \dots, N\}\}$.}  According to 
Remark~\ref{rem:cramer}, each $\mathbf u_I$ solves~\eqref{eq:ineqs}.  
Moreover, $\mathbf u_I$ satisfies with equality the subsystem 
\begin{equation*}
\left\{
\begin{array}{l}
f_j \ge 0 \quad (j \notin I), \\
u_i \ge 0 \quad (i \in I)
\end{array}
\right.
\end{equation*}
of~\eqref{eq:ineqs} of rank $N$.  Consequently, each $\mathbf u_I$ is a vertex 
of $P$.  We must prove that $P$ has no other vertices.

In the hyperplane $u_N = 0$ consider the polyhedron $P'$ being the solution 
set of
\begin{equation}
\label{eq:3s02}
\left\{
\begin{array}{l}
\tilde f_i \ge 0 \quad (i = 1, \dots, N - 1), \\
u_i \ge 0 \quad (i = 1, \dots, N - 1),
\end{array}
\right.
\end{equation}
where $\tilde f_i$ is the restriction of $f_i$ to the hyperplane.  By the 
induction assumption, $P'$ is an $(N-1)$-dimensional polytope with vertices 
$\{\mathbf u_I \colon I \ni N\}$.  We claim that it is a facet of $P$.  
Indeed, let $\widetilde P$ be the face of $P$ on the hyperplane $\{u_N = 
0\}$.  On this hyperplane $\widetilde P$ is given by
\begin{equation}
\label{eq:fphi1}
\left\{
\begin{array}{l}
\tilde f_i \ge 0 \quad (i = 1, \dots, N), \\
u_i \ge 0 \quad (i = 1, \dots, N - 1),
\end{array}
\right.
,
\end{equation}
and the inclusion $\widetilde P \subset P$ is immediate.  On the other hand, 
by the induction assumption, any vertex of $P'$ is one of the points $\mathbf 
u_I$ lying in the hyperplane, so it is a vertex of $P$ and also of $\widetilde P$.  
Consequently, $P' \subset \widetilde P$.  Thus, we have $P' = \widetilde P$, 
and $P'$ is an $(N-1)$-dimensional face of $P$.  The vertex $\mathbf 
u_\varnothing$ of $P$ does not belong to the hyperplane $\{u_N = 0\}$, so $P$ 
has dimension $N$, and $P'$ is its facet.

Likewise, the interection of $P$ with any hyperplane $u_i = 0$ is a facet 
of~$P$ having the vertices $\{\mathbf u_I \colon I \ni i\}$.

Let $\mathbf v = (v_1, \dots, v_N)$ be a vertex of $P$.  Then $\mathbf v$ 
satisfies with equalities a subsystem of~\eqref{eq:ineqs} of rank~$N$ 
consisting of $N$ inequalities.  If this subsystem is $f_i \ge 0$ ($i = 1, 
\dots, N$), then $\mathbf v = \mathbf u_{\varnothing}$.  Otherwise, we have 
$v_i = 0$ for some $i$, so $\mathbf v$ lies in the hyperplane $u_i = 0$ and by 
the above, coincides with one of $\mathbf u_I$.  Thus, the set of vertices of 
$P$ is exactly $\{\mathbf u_I \colon I \subset \{1, \dots, N\}\}$.

\emph{Step 4.  Positivity of the determinants implies that the facets of $P$ 
are the intersections of $P$ with the hyperplanes $u_i = 0$ and $f_i = 0$, $i 
= 1, \dots, N$}.  As $P$ is given by~\eqref{eq:ineqs}, each facet of $P$ is 
the intersection of $P$ with a hyperplane of the form $u_i = 0$ or $f_i = 0$.  
Conversely, each intersection of this form is a facet of $P$.  Indeed, we have 
already checked this in the case of the hyperplanes $u_i = 0$.  Now let $P' = 
P \cap \{f_1 = 0\}$.  The face $P'$ contains, among others, the vertices 
$\mathbf u_I$, where $I \not\ni 1$ and $I \ni N$.  By the induction 
assumption, these are precisely the vertices of an $(N - 2)$-dimensional facet 
of $P'' = P \cap \{u_N = 0\}$.  Thus, $\dim P' \ge N - 2$.  On the other hand, 
$P'$ contains $\mathbf u_\varnothing$, which is affinely independent of 
$\{\mathbf u_I \colon I \ni N\}$, so actually $\dim P' = N - 1$, as claimed.

\emph{Step 5.  Positivity of the determinants implies that $P$ is 
combinatorially isomorphic to a cube.}  Indeed, considering the cube as the 
solution set of
\begin{equation*}
\left\{
\begin{array}{l}
1 - u_i \ge 0 \quad (i = 1, \dots, N), \\
u_i \ge 0 \quad (i = 1, \dots, N),
\end{array}
\right.
,
\end{equation*}
we see that the mapping $\mathbf u_I \mapsto (\alpha_1, \dots, \alpha_N)$, 
where
\begin{equation*}
\alpha_i =
\begin{cases}
0, & \text{if } i \in I, \\
1 & \text{otherwise}
,
\end{cases}
\end{equation*}
preserves facets.

\emph{Step 6.  The geometric properties of $P$ imply the positivity of the 
determinants.} Conversely, assume that solution set of~\eqref{eq:ineqs} is a polytope 
with vertices $\{\mathbf u_I \colon I \subset \{1, \dots, N\}\}$ 
combinatorially isomorphic to a cube.  Observe that given $i$, the set of vertices of the facet $P \cap \{u_i = 0\}$ is $\{\mathbf u_I \colon I \ni i\}$.  Consequently, if $i \in I$, the vertex $\mathbf u_{\mathsf CI}$ does not belong to this facet, so $u_{\mathsf CIi} > 0$ whenever $i \in I$.  Likewise, 
 $f_j(\mathbf u_{\mathsf CI}) > 0$ whenever $j \notin I$.  Now it follows 
from~\eqref{eq:ufu} that all the determinants $\Delta_I$ and $\Delta_{I,j}$ 
have the same sign.  As the facet $P' = P \cap \{u_N = 0\}$ enjoys analogous 
geometric properties, by induction we see that the determinants are actually positive.
\end{proof}

\begin{proof}[Proof of Lemma~\ref{lem:obs}]
We fix an admissible set of coefficients and $j$ and prove that
\begin{equation}
  \label{eq:obs0}
\sup
\{
\min_{i} f_i(\mathbf u)
\colon
\mathbf u = (u_1, \dots, u_N) \ge 0,
u_j \le \sigma, f_j(\mathbf u) \le \sigma
\}
\le - \sigma
\end{equation}
with some $\sigma$ independent of the coefficients.  For definiteness, assume 
that $j = N$.

As above, denote the polytope given by~\eqref{eq:ineqs} by $P$.

It follows from Remark~\ref{rem:cramer} that there exists $c > 0$ independent 
of the coefficients such that $P$ has no vertices in the open slab $\{0 < u_N 
< 2c\}$.  In other words, all the vertices of the polytope $P_{2c} := P \cap 
\{0 \le u_N \le 2c\}$ lie on the hyperplanes $u_N = 0$ and $u_N = 2c$.  It is 
easy to check that any point of $P_{2c}$ belonging to the hyperplane $u_N = c$ 
is the midpoint of a line segment with the endpoints on the facets $P_{2c}' = 
P_{2c} \cap \{u_N = 0\}$ and $P_{2c}'' = P_{2c} \cap \{u_N = 2c\}$ of $ P$, the 
former being also a facet of $P$.

By Remark~\ref{rem:cramer}, there exists $c' > 0$ independent of the 
coefficients such that $f_N \ge 2c'$ on each vertex of the facet $P \cap
\{u_N = 0\} = P_{2c}'$, so $f_N \ge 2c'$ on $P_{2c}'$.  Also $f_N
\ge 0$ on $ P_{2c}'' \subset P$.  Consequently, $f_N \ge c'$ on
$P_{2c} \cap \{u_N = c\} = P \cap \{u_N = c\}$.

By the above, all the vertices of the polytope $P_c = P \cap \{0 \le
u_N \le c \}$ lie in the halfspace $f_N \ge c'$, therefore so
does the polytope itself. In other words, the polytope $P_c$ is the solution set of
\begin{equation*}
  \left\{
    \begin{array}{l}
      f_i \ge 0 \quad (i = 1, \dots, N - 1),
      \\
      u_i \ge 0 \quad (i = 1, \dots, N),
      \\
      u_N \le c,
    \end{array}
  \right.
\end{equation*}
and this system implies the inequality $f_N - c' \ge 0$.  By the
Minkowski--Farkas theorem, there exist nonnegative $\alpha_i$,
$\beta_i$, $\gamma$, and $\delta$ such that
\begin{equation}
\label{eq:obs02}
  f_N - c'
  = \sum_{i = 1}^{N-1} \alpha_i f_i
  + \sum_{i = 1}^{N} \beta_i u_i
  + \gamma (c - u_N)
  + \delta
  .
\end{equation}
Generally speaking, representation~\eqref{eq:obs02} is not unique, but
we claim that possible values of $\alpha_i$ are uniformly bounded with
respect to admissible coefficients.  Indeed, plugging the zero vertex
into~\eqref{eq:obs02}, we get
\begin{equation*}
  f_N(0) - c'
  = \sum_{i = 1}^{N-1} \alpha_i \tilde f_i(0)
  + \gamma c
  + \delta
  ,
\end{equation*}
whence
\begin{equation*}
  \sum_{i = 1}^{N-1} \alpha_i f_i(0)
  \le
  f_N(0) - c'
  .
\end{equation*}
By Remark~\ref{rem:cramer}, the right-hand side is bounded, and the values of 
$f_i(0)$ in the left-hand side are bounded away from 0. Consequently, there 
exists $C>0$ independent of admissible coefficients such that
\begin{equation*}
  \alpha_i \le C
\end{equation*}
for any possible choice of $\alpha_i$, as claimed.

Now write~\eqref{eq:obs02} in the form
\begin{equation*}
  \sum_{i=1}^{N-1} \alpha_i f_i
  =
  - \frac{c'}{2}
  + \left(
    f_N - \frac{c'}{2}
  \right)
  - \sum_{i=1}^N \beta_i u_i
  - \gamma(c - u_n)
  - \delta
\end{equation*}
and observe that whenever $u_N \le c$ and $f_N \le c'/2$, we have
\begin{equation*}
  \sum_{i=1}^{N-1} \alpha_i f_i
  \le
  - \frac{c'}{2}
  ,
\end{equation*}
whence
\begin{equation*}
\sum_{f_i < 0} f_i
\le -\frac{c'}{2C}
.
\end{equation*}
There are at most $N$ summands in the left-hand side, so for any
$\mathbf u$ satisfying said requirements there exists $f_i$ such that
$f_i(\mathbf u) \le - c'/(2CN)$.  Thus, \eqref{eq:obs0} is valid
with $\sigma = \min\{c, c'/2, c'/(2CN)\}$, which is clearly independent of 
particular choice of admissible coefficients.
\end{proof}

The following lemma is the first step towards proving Lemma~\ref{lem:gstab}.  
It ensures that $v$ does not vanish at infinity.

\begin{lemma}
\label{lem:limv}
\begin{equation}
\label{eq:limv}
\lim_{
\substack{|\mathbf u| \to \infty\\
\mathbf u \in \mathbb R^N_+
}}
v = \infty
\end{equation}
uniformly with respect to admissible coefficients.
\end{lemma}

\begin{proof}
Assume the contrary: there exist $C > 0$, a sequence of admissible sets of 
coefficients $\{(A_n, \mathbf m_n)\}$ and a sequence $\{\mathbf u_n\} \subset 
\mathbb R^N$ such that $\mathbf u_n \ge 0$, $\mathbf u_n \to \infty$, and $0 
\le v_n (\mathbf u_n) \le C$, where $v_n$ is corresponding to the coefficients 
$(A_n, \mathbf m_n)$.  By $f_{in}$ denote the affine functions corresponding 
to $(A_n, \mathbf m_n)$, and by $f'_{in}$, associated linear functionals.

By Assumption~\ref{a:bounded}, the sequence $\{(A_n, \mathbf m_n)\}$ is 
bounded.  Without loss of generality, $(A_n, \mathbf m_n) \to (A_*, \mathbf 
m_*)$, the limiting coefficients also being admissible.  Let $f_{i*}$ be the 
corresponding affine functions, and $f'_{i*}$ be the associated linear 
functions.

Without loss of generality,
\begin{equation}
\label{eq:contu01}
u_{in} = b_i \tau_n + o(\tau_n)
,
\end{equation}
where $b_i \ge 0$ is finite, $\mathbf b = (b_1, \dots, b_N) \ne 0$, $\tau_n 
\to \infty$.  We have:
\begin{equation*}
f_{in}(\mathbf u_n) = m_{in} + f'_{in}(\mathbf u_n)
= m_{in} + f'_{in}(\mathbf b) \tau_n + o(\tau_n)
= f'_{in}(\mathbf b) \tau_n + o(\tau_n)
\end{equation*}
As $f'_{in} \to f'_{i*}$ and the sequences $\{m_{in}\}$ are bounded, we obtain
\begin{equation}
\label{eq:contu02}
f_{in}(\mathbf u_n) = f'_{i*}(\mathbf b) \tau_n + o(\tau_n)
.
\end{equation}

Plugging representations \eqref{eq:contu01} and \eqref{eq:contu02} into the 
right-hand side of~\eqref{eq:vg}, we obtain:
\begin{equation*}
\sum_{i=1}^N u_{in} |f_{in}(\mathbf u_n)|^p
=
\left(
\sum_{i=1}^N b_i |f'_{i*}(\mathbf b)|^p
\right)
\tau_n^{p+1}
+ o(\tau_n^{p+1})
.
\end{equation*}
Here the leading coefficient does not vanish.  If it did, we would have $b_i 
|f'_{i*}| = 0$ for any $i$, so $\mathbf b$ would solve the linear system
\begin{equation*}
\left\{
\begin{array}{l}
f'_{i*}(\mathbf b) = 0 \quad (i \in I), \\
b_i = 0 \quad (i \notin I)
\end{array}
\right.
\end{equation*}
for some $I \subset \{1, \dots, N\}$.  But due to Assumption~\ref{a:det-a} 
this system only has the trivial solution.

Thus, the right-hand side of~\eqref{eq:vg} grows as $\tau_n^{p+1}$.  On the 
other hand, a trivial verification shows that the left-hand side 
of~\eqref{eq:vg} is $O(\tau_n^p)$, a contradiction.
\end{proof}

\begin{proof}[Proof of Lemma~\ref{lem:gstab}]
Suppose, contrary to our claim, that there exist $\varepsilon > 0$, a sequence 
$\{(A_n, \mathbf m_n)\}$ of admissible coefficients, and a sequence $\{\mathbf 
u_n = (u_{1n}, \dots, u_{Nn})\} \subset \mathbb R_+$ such that
\begin{equation*}
v_{n}(\mathbf u_n) \to 0
\end{equation*}
and
\begin{equation}
\label{eq:gstab02}
| \mathbf u_n - \mathbf{u}_{In} | \ge \varepsilon
\quad (I \ne \varnothing)
,
\end{equation}
where $v_n$ and $\mathbf u_{In}$ corresponds to the coefficients $(A_n, 
\mathbf m)$.  Due to Assumptions~\ref{a:bounded}--\ref{a:det-am} (see also 
Remark~\ref{rem:cramer}) there is no loss of generality in assuming that 
$(A_n, \mathbf m) \to (A_*, \mathbf m_*)$, where the limit is also admissible, 
and $\mathbf{u}_{I_n} \to \mathbf{u}_{I*}$ for each $I \ne \varnothing$, where 
the limit satisfies \eqref{eq:vi}.  Thus, if $v_*$ corresponds to the limiting 
coefficients, we have $v_* (\mathbf{u})= 0$ if and only if $\mathbf u = 
\mathbf{u}_{I*}$ for some $I \ne \varnothing$.

Due to Lemma~\ref{lem:limv}, $\{\mathbf u_n\}$ is bounded, and we can assume 
that $\mathbf u_n \to \mathbf u_* = (u_{1*},\dots,u_{N*})\in \mathbb R^N_+$.  
Passing to the limit in~\eqref{eq:gstab02}, we get
\begin{equation}
\label{eq:gstab03}
| \mathbf u_* - \mathbf{u}_{I*} | \ge \varepsilon
\quad (I \ne \varnothing)
.
\end{equation}
By \eqref{eq:uvu},
\begin{equation*}
\min_i u_{i*} = \lim_{n\to\infty} \min_i u_{in} \le
\lim_{n\to\infty} v_{n}(\mathbf{u}_n) = 0,
\end{equation*}
so $\mathbf u_* \ne \mathbf u_\varnothing$ (Remark~\ref{rem:cramer}).  Thus, 
$g_*(\mathbf u_*) \ne 0$ and we can pass to the limit:
\begin{equation*}
v_*(\mathbf u_*) = \lim_{n \to \infty} v_{n}(\mathbf u_n) = 0
.
\end{equation*}
This and the fact that $\mathbf u_*$ is nonnegative implies $\mathbf u_* = 
\mathbf{u}_{I*}$ for some $I$, which
contradicts~\eqref{eq:gstab03}.
\end{proof}
\begin{proof}[Proof of Corollary~\ref{cor:genlimv}]
Given $I \ne \varnothing$, consider the norm
\begin{equation*}
| \mathbf{u} |_I = \sum_{i \in I} |u_i| + \sum_{j\notin I} |f'_j(\mathbf{u})|
,
\end{equation*}
which implicitly depends on the choice admissible coefficients.  Observe that
\begin{equation}
\label{eq:gstab10}
C_1 |\mathbf{u}| \le |\mathbf{u}|_I \le C_2|\mathbf{u}|
,
\end{equation}
where $C_1$ and $C_2$ depend on $\kappa$ but not on admissible coefficients.  
Indeed, letting for simplicity $I = \{1, \dots, r\}$, we have $|\mathbf{u}|_I 
= |A_I\mathbf{u}|$, where
\begin{equation*}
A_I =
\begin{bmatrix}
1     \\
& \ddots &&& \mathbf{0}\\
& & 1         \\
& & & -a_{r+1, r+1} & \dots & -a_{rN} \\
& \mathbf{0} & & \vdots     & \ddots & \vdots \\
& & & -a_{N, r+1}  & \dots  & -a_{NN} \\
\end{bmatrix}
\end{equation*}
and it follows from Assumptions~\ref{a:bounded} and~\ref{a:det-a} that the 
norms $\|A_I\|$ and $\|A_I^{-1}\|$ are bounded uniformly with respect to 
admissible coefficients.

By Assumption~\ref{a:bounded}, there exits $C$ depending on $\kappa$ such that 
for any admissible coefficients $\|f_{i}'\| \le C$ for any $i$, where 
$\|\cdot\|$ is the norm of a linear functional on $\mathbb R^N$.

Take $\varepsilon > 0$.  By Lemma~\ref{lem:gstab} there exists $\delta > 0$ 
independent of admissible coefficients such that whenever $v(\mathbf{u}) < 
\delta$, we have
\begin{equation*}
\min_{I \ne \varnothing} |\mathbf{u} - \mathbf{u}_{I} | <
\frac{\varepsilon}{C+C_2}
.
\end{equation*}
Take $I \ne \varnothing$ such that
\begin{equation*}
|\mathbf{u} - \mathbf{u}_{I} | <
\frac{\varepsilon}{C+C_2}
,
\end{equation*}
then \eqref{eq:genlimv1} holds.

By Remark~\ref{rem:cramer}, for any $i \in I$ we have 
$f_{i}(\mathbf{u}_{I}) \ge c$ with $c$ independent of admissible 
coefficients, so
\begin{equation*}
f_{i}(\mathbf{u})
= f_{i}(\mathbf{u}_{I} ) + f_{i}'( \mathbf{u} - 
\mathbf{u}_{I} )
\ge c - \varepsilon.
\end{equation*}
Without loss of generality, $\varepsilon < c/2$, so
\eqref{eq:genlimv2} holds with $\sigma = c/2$.
\end{proof}

\end{document}